\documentclass[a4paper,12pt]{amsart}

\usepackage{amssymb,amsmath,latexsym,amsthm}
\usepackage{datetime}
\usepackage{xcolor}
\usepackage{tikz}
\usepackage{float}
\usetikzlibrary{calc,arrows.meta}

\usepackage[pdfauthor={Dimitrios Chatzakos},
pdftitle={},
            pdfkeywords={L-functions},
             pdfcreator={Pdflatex}]
{hyperref}
\hypersetup{colorlinks=true}
\usepackage{xcolor}
\usepackage{enumerate}

\usepackage[english]{babel}
\usepackage{todonotes} %To add notes

\theoremstyle{plain}
\newtheorem{theorem}{Theorem}[section]

\newtheorem{lemma}[theorem]{Lemma}

\newtheorem{corollary}[theorem]{Corollary}
\newtheorem{proposition}[theorem]{Proposition}

\newtheorem*{rem*}{Remark}

\theoremstyle{definition}

\newtheorem{remark}[theorem]{Remark}

\def\MR#1{\href{http://www.ams.org/mathscinet-getitem?mr=#1}{MR#1}}

\numberwithin{equation}{section}

\DeclareMathOperator{\tr}{tr}

\newcommand{\Z}{{\mathbb Z}}

\newcommand{\R}{{\mathbb R}} 
\newcommand{\C}{{\mathbb C}}

\def \slnz  {{\hbox{SL}(n, {\mathbb Z})} }

\def \fpsi {{\hat{\psi}_{\delta}}}

\def \e     {{\epsilon}}
\def \k     {{\kappa}}
\def \m     {{\mu}}
\def \r     {{\rho}}
\def \t     {{\tau}}
\def \s     {{\sigma}}
\def \b     {{\beta}}

\title[Higher rank lattice counting]{Pointwise improvement of the error term in the higher rank hyperbolic lattice counting problem}

\author[Chatzakos]{Dimitrios Chatzakos}
\address{
   Department of Mathematics,
   University of Patras,
   26 504, Patras,
   Greece
}
\email{dchatzakos@math.upatras.gr}

\author[Dimakis]{Panagiotis Dimakis}
\address{
Department of Mathematics, 
University of Maryland,
College Park 20740, MD, 
USA.
}
\email{pdimakis12345@gmail.com}

\author[Katsivelos]{Christos Katsivelos}
\address{
   Department of Mathematics,
   University of Patras,
   26 504, Patras,
   Greece
}
\email{c.katsivelos@upatras.gr}

\date{\today}


\begin{document}

\begin{abstract}
We prove an upper bound for the error term in the higher rank hyperbolic lattice counting problem for ${\hbox{SL}(n, \mathbb{Z})}$. Blomer and Lutsko achieved recently a significant improvement of previous known results by Duke--Rudnick--Sarnak and Gorodnik--Nevo--Yehoshua. We deduce an improved upper bound for the error which holds for almost all pairs $z,w \in {\hbox{SL}(n, \mathbb{R})}$. We also derive a $\Omega$-result for the error term, in the spirit of Phillips--Rudnick. 
\end{abstract}

\maketitle

\section{Introduction}

\subsection{Background and general setting} %\label{subsection1.1}

The hyperbolic lattice counting problem is an important problem in the analytic theory of automorphic forms. For the case of cofinite lattices $\Gamma$ acting on classical hyperbolic spaces $\mathbb{H}^n$ this problem has been extensively studied, using the spectral theory of automorphic forms in Lie groups of rank one, see \cite{biro, biro2, cham1, cham2, chatz, chatz3, chatzakosdimakis, cherubini, cherubinikatsivelos, gunther, Iwaniec2002, katsivelos, laxphillips, patterson, PetridisRisager:2018a, PhillipsRudnick:1994, selberg}. This problem can also be understood as a counting problem for matrices in $\Gamma$ ordered by their Frobenius norm. 

In this work we are interested in the lattice counting problem for the higher rank case. This problem has also attracted considerable interest, see \cite{blomerlutsko, Dukerudnicksarnak, EskinMcMullen, GNY}. For $n \geq 3$ let $G={{\hbox{SL}(n,{\mathbb R})} }$, $K = {{\hbox{SO}(n,{\mathbb R})} }$ the maximal compact subgroup of orthogonal matrices and $\Gamma = {{\hbox{SL}(n,{\mathbb Z})} }$. In analogy with the rank one case, this problem can be studied using the spectral theory of automorphic forms on the symmetric space
\begin{equation*}
    X=\Gamma\backslash G/ K,
\end{equation*}
where $G/K$ is the \textit{generalized hyperbolic space} of rank $n-1$. Consider the Iwasawa decomposition $G=KAN=NAK$, where $N$ is the group of unipotent upper triangular matrices of determinant $1$, $A$ the group of diagonal matrices with positive entries and determinant $1$ and $K$ as above. We equip these groups with their respective Haar measures, normalized such that 
\begin{equation*}
\hbox{vol}(K)=1\quad \text{and denote}\quad V=\hbox{vol}(X).
\end{equation*}

As in \cite{blomerlutsko}, for any two fixed points $z,w \in G$ we are interested in the counting function given by
\begin{equation}\label{counting function}
    N_n(T;z,w)=\#\{\gamma\in\Gamma:\|z^{-1}\gamma w\|_F\leq T\}, \quad  
    \end{equation}
where by $\|\cdot\|_F$ we denote the Frobenius norm
\begin{equation*}
\|g\|_F=\sqrt{\tr{(gg^T)}}.
\end{equation*}
Instead of two points $z,w$ on the space $X$, we count representatives over the group $G$. This does not affect the counting due to the left $\Gamma$-invariance and also bi-$K$-invariance.  Moreover, the counting function $N_n(T;z,w)$ is well defined and finite for every $T >0$, since a Frobenius ball in $G$ is compact and $\Gamma$ is discrete.

We are interested in estimating the asymptotic behavior of \eqref{counting function} as $T \to \infty$. For that reason, let 
\begin{eqnarray} \label{sharperror}
E_n(T;z,w) :=  N_n (T;z,w)  - c_n T^{n(n-1)},
\end{eqnarray}
where the constant $c_n$ is given by
\begin{eqnarray*}
c_n = \frac{\pi^{n^2/2}}{\Gamma (\frac{n^2-n+2}{2}) \Gamma(\frac{n}{2}) \zeta(2) \cdot \cdot \cdot \zeta(n)}.
\end{eqnarray*}
Improving on previous results, Blomer and Lutsko proved strong upper bounds for the error term \eqref{sharperror}. Their main results are the following. 
\begin{theorem}[Blomer--Lutsko \cite{blomerlutsko}] \label{blomerlutsko1}
Let $n \geq 3$ and $z,w \in G$. Then for every $\epsilon >0$, as $T \to \infty$:
\begin{eqnarray*}
 E_n (T;z,w)  = O_{n, z, w,\epsilon} \left(T^{n(n-1) - \delta_n +\epsilon}\right),
\end{eqnarray*}
where the exponent $\delta_n = \delta_{n}^{BL} $ is given by $\delta_3^{BL} = 1$, $\delta_4^{BL} = 6/5$, whereas for $n\geq 5$:
\begin{equation*} 
\delta_{n}^{BL} = \frac{n}{2n-1 - \sqrt{2n^2 - 10n -4}} = 1 + \frac{1}{\sqrt{2}} + O \left( \frac{1}{n}\right).
\end{equation*}
\end{theorem}

Moreover, using a local Weyl's law of M\"uller \cite{muller}, they achieved a better upper bound for the local average of the error term $E_n(T;z,w)$ over compact subsets of $G$. To state their second result, let $\Omega\subset G$ be a compact subset with respect to the Haar measure $d\mu(z)$.
\begin{theorem}[Blomer--Lutsko \cite{blomerlutsko}]      \label{blomerlutsko2}
Let $n \geq 3$. Then for every $\epsilon >0$, as $T \to \infty$:
\begin{equation*}
 \int_{\Omega} E_n  (T;z,z)  \ d \mu(z) = O_{n, \Omega,\epsilon} \left(T^{n(n-1) - \delta_n + \epsilon}\right)
\end{equation*}
with
\begin{equation} \label{blexponent2}
\delta_n = \frac{2 (n^2-2n)}{n^2+3n-2} = 2 + O \left( \frac{1}{n}\right).
\end{equation}
\end{theorem}

\begin{remark}
The pointwise upper bound of Theorem \ref{blomerlutsko1} was a significant strengthening of previous known upper bounds. The first such upper bound was obtained by Duke--Rudnick--Sarnak \cite{Dukerudnicksarnak} with the exponent

\begin{equation*}
\delta_{n}^{DRS} = \frac{1}{n+1}.
\end{equation*}
This was slightly improved by Gorodnik--Nevo--Yehoshua \cite{GNY} with a pointwise improvement of asymptotic size
\begin{equation*}
\delta_{n}^{GNY} \sim 2 \cdot \delta_n^{DRS}.
\end{equation*}
\end{remark}

\begin{remark}
Local average results analogous to Theorem \ref{blomerlutsko2} have been studied in rank one \cite{cherubinikatsivelos, katsivelos, PetridisRisager:2018a}, using estimates for the Quantum variance of Laplace eigenfunctions as well as estimates for spectral exponential sums appearing in the study of the Prime geodesic theorem. 
\end{remark}

\subsection{Our results}

In this work we improve the pointwise upper bound of Blomer and Lutsko for almost all pairs $z,w$. Our main result is the following.

\begin{theorem} \label{ourmainthm}
Let $n \geq 3$ and fix $w \in G$. Then, there exists a full measure subset $B_w \subset G$, such that for every $z \in B_w$ and for every $\epsilon>0$ the following bound 
\begin{eqnarray*}
 E_n (T;z,w)  = O_{n,z,w,\epsilon} \left(T^{n(n-1) - \min\left\{n/2,2\right\} + \epsilon}\right)
\end{eqnarray*}
holds as $T \to \infty$, uniformly for all sufficiently large $T$. The analogous assertion holds for every fixed $z \in G$. In particular,
\begin{equation*}
E_3 (T;z,w) = O \left(T^{9/2 +\epsilon}\right), \ \ \ E_4 (T;z,w)= O \left(T^{10 +\epsilon}\right)
\end{equation*}
for almost all $z,w$. 
\end{theorem}
In comparison, Theorem \ref{blomerlutsko1} gives errors $E_3 (T;z,w)= O(T^{5+\epsilon})$ and $E_4 (T;z,w) = O(T^{54/5 +\epsilon})$ respectively, for every fixed pair $z,w$. For every $n \geq 4$, the power saving in the upper bound takes the form $T^{n(n-1) -2 +\epsilon}$. In \cite{blomerlutsko} it is predicted that this power saving may be essentially optimal, and Theorem \ref{blomerlutsko2} supports this expectation even for $n=3$. In the rank one case, several averaging results, higher moments and $\Omega$ results have been studied for the error term in the hyperbolic lattice counting problem, see for instance \cite{cham2, chatz, PhillipsRudnick:1994}. In the next result we prove an $\Omega$-estimate for the error term in the higher rank case, confirming the prediction from \cite{blomerlutsko}.

\begin{theorem} \label{omegaresult}
For $n \geq 3$ the error term $E_n (T;z,w)$ cannot be very small for \textit{every} fixed pair $z,w$. In particular we have
\begin{equation*}
E_n (T;I_n,I_n)=  \Omega \left(T^{n(n-1)-2}\right),
\end{equation*}
where by $I_n$ we denote the identity matrix.
\end{theorem}

The proof of Theorem \ref{omegaresult} is quite elementary and arithmetic in spirit. This $\Omega$-result can be compared with the known $\Omega$-results for the classical Euclidean and hyperbolic lattice counting problems in large dimensions, see for instance \cite{PhillipsRudnick:1994}.

\begin{remark}
For $n \geq 3$, the Frobenius norm of $g^{-1}$ need not equal that of $g$. Consequently, the kernel \eqref{auto kernel count} (defined in terms of \eqref{automorphickernel}) lacks symmetry in $z,w$.
The orientation in \eqref{automorphickernel} determines the conjugations in the pre-trace formula. Keeping
that orientation avoids importing a rank-one symmetry which higher-rank norm balls do not possess.   
\end{remark}

\begin{remark}
The argument of our proof works for almost all pairs of points $(z,w)$, hence it cannot guarantee the bound for any \textit{specific} preassigned pair of points. The set $B_w$ in Theorem \ref{ourmainthm} does in general depend on $w$, and as such, does not produce any explicit point $z \in B_w$, a single full-measure set valid for all $w$ or any information for the diagonal case $z=w$.
\end{remark}

\begin{remark}
Our main Theorem \ref{ourmainthm} is similar in spirit with \cite{Jana1, Jana2}, where optimal almost-everywhere approximation results for dense rational orbits are studied (see also \cite{GGN}).
\end{remark}

\subsection{Strategy of proof} We follow the standard strategy in the study of lattice counting problems, which starts with the use of the pre-trace formula (i.e. the spectral theorem for point pair invariant kernels) for smooth approximations of the counting function. We follow the strategy of Blomer and Lutsko, combined with a new maximal estimate (Proposition \ref{maximal estimate}) and the first Borel-Cantelli lemma. We refer to \cite{Sogge:1993} for a more concrete introduction to maximal estimates and their applications, and to \cite{chatzakosdimakis, dimakis} for other recent applications of analogous maximal estimates.

\subsection{AI Declaration} 
We got the idea of using maximal estimates for bounding the error term from \cite{chatzakosdimakis}. The main idea in \cite{chatzakosdimakis} was obtained through extensive interactions with ChatGPT-5.6 Sol. The authors formulated the project, evaluated and revised the model’s suggestions, independently verified all mathematical arguments, and wrote the final manuscript. The authors take full responsibility for the correctness of all statements and proofs in this paper.

\subsection{Acknowledgements} D.C. and C.K. were supported by the Hellenic Foundation for Research and Innovation (H.F.R.I.) under the “3rd Call for H.F.R.I. Research Projects to support Faculty Members \& Researchers” (Project Number: 25622). We would like to thank Yiannis Petridis and Polyxeni Spilioti for useful conversations.

\section{Background material}

\subsection{Notation} We fix the following notation which will be used throughout the paper:
\begin{equation*}
    d=n(n-1),\quad r=n-1,\quad b=d-\frac{n}{2}.
\end{equation*}
Our approach uses the pre-trace formula applied to smooth approximations of the
\begin{equation} \label{ktdefinition}
k_T = \chi_{B_T},
\end{equation}
where $B_T =\{g\in G:\|g\|_F\leq T\}$ is the $\|\cdot\|_F$-ball of radius $T$. Let us also set $M_n(T) = c_n T^d$ and recall from \eqref{sharperror} that $E_n(T;z,w) = N_n(T;z,w) - M_n(T)$. We also distinguish the contribution of the constant eigenfunction from the main term:
\begin{equation}\label{setting}
    E_0(T;z,w) =N_{n}(T;z,w)-M_0(T), 
\end{equation}
with 
\begin{equation} \label{m0}
M_0(T) = V^{-1}\hbox{vol($B_T$)}. 
\end{equation}
The replacement of $M_0$ by $M_n$ will be justified by a residue calculation, including a differentiated remainder estimate.

\subsection{Spectral input and the spherical transform}

The material covered in this section is standard; for more details we refer to \cite{blomerlutsko, goldfeld, helgason, mw1, mw2} and references therein.

Let $\mathfrak{a}$ be the Lie algebra of $A$. The Iwasawa projections $A,H:G\rightarrow\mathfrak{a}$ are characterized by 
\begin{equation*}
    g\in K\exp(H(g))N=N\exp(A(g))K.
\end{equation*}
Identify
\begin{equation*}
    \mathfrak{a}^\ast=\{\t\in\R^n:\sum_{i=1}^{n}\t_i=0\},\quad \quad \mathfrak{a}_{\C}^{\ast}=\{\m\in\C^n:\sum_{i=1}^{n}\m_i=0\}.
\end{equation*}
For each spectral parameter $\m\in\mathfrak{a}_{\C}^{\ast},$ we have that 
\begin{equation}\label{spectral parameter conditions}
    \{\m_j\}=\{-\overline{\m_j}\},\quad \mu\in C_\r+i\mathfrak{a}^\ast,
\end{equation}
with $C_\r$ being the convex hull of the Weyl orbit of $\rho$ and $\rho=(\frac{n-1}{2},\frac{n-3}{2},...,\frac{3-n}{2},\frac{1-n}{2})$. We will use the notation 
\begin{eqnarray}\label{spectral parameters}
    \mu &=& m+i\t,\quad m\in C_\r,\,\t\in\mathfrak{a}^{\ast}, \nonumber\\
    \b &:=& \b(\m) =\|\Re\m\|_\infty, \\
    U &:=& U(\m)=1+\|\t\|_\infty. \nonumber
\end{eqnarray}
Since $C_\rho$ is the convex hull of a finite set it is compact. We deduce that $\|\Re\m\|_\infty$ is bounded and also
\begin{equation*}
    U(\m)\asymp_n 1+\|\m\|_\infty.
\end{equation*}

For $\mu\in\mathfrak{a}_\C^\ast$, following the convention of \cite{blomerlutsko}, the elementary spherical function is  defined by 
\begin{equation*}
    \phi_{\mu}(g)=\int_{K}\exp \left((\m-\r) \, H(gk) \right)\,dk=\int_{K}\exp \left((\m+\r) \, A(kg)\right)\,dk.
\end{equation*}
This is bi-$K$-invariant and depends only on the Weyl orbit of $\m$. The normalization $\hbox{vol}(K)=1$ yields for all $g\in G$:
\begin{eqnarray*}%\label{sphericalproperties}
  \phi_\r (g) &=& 1,  \nonumber  \\
  |\phi_\m(g)| &\leq& 1 \ \ \text{for all} \  \m\in\mathfrak{a}_\C^\ast.
  \end{eqnarray*}
  
For a compactly supported, integrable, bi-$K$-invariant function $f:G\rightarrow\C$ define its spherical transform by
\begin{equation}\label{spherical 1 step}
    h_{f}(\mu)=\int_{G}f(g)\phi_{-\m}(g)\,dg.
\end{equation}
Alternatively this can be obtained in two steps, first using the Abel transform
\begin{equation*}
    (\mathcal{A}_f)(a)=\int_{N} e^{\r \log a}f(an)\,dn
\end{equation*}
and then by composing with the Fourier transform 
\begin{equation*}
    h_{f}(\m)=\int_{A}(\mathcal{A}_f)(a)e^{-\m\log a}\,da.
\end{equation*}

For two  integrable, compactly supported, bi-$K$-invariant functions $f_1,f_2$ we use the standard convolution convention
\begin{equation*}
    (f_1\ast f_2)(x)=\int_{G}f_1(xg^{-1})f_2(g)\,dg.
\end{equation*}
The spherical transform is multiplicative, in the sense 
\begin{equation}\label{convolution spherical}
    h_{f_1\ast f_2}=h_{f_1}\cdot h_{f_2}.
\end{equation}
For more details on the spherical transform, we refer to \cite{helgason}.

\subsection{The spectral measure and Plancherel synthesis}

Let $(\Omega,d\omega)$ denote a normalized parameter space  with multiplicities and let $u_\omega$ denote the corresponding generalized eigenfunctions. Integration over $\Omega$ includes a sum over an orthonormal basis of the discrete spectrum, together with the appropriate Eisenstein integrals. In particular $d\omega$ is a positive measure. We write $\omega=0$ for the constant constituent and normalize it such that 
\begin{equation*}
    \m_0=\rho,\quad u_0=\frac{1}{V^{1/2}}.
\end{equation*}
Removing this constituent we denote the remaining parameter space by $\Omega'.$

Take $P:L^2(X)\rightarrow\mathcal{C}$ to be the Plancherel analysis map, where $\mathcal{C}$ denotes the coefficient Hilbert space. Then the adjoint $S=P^\ast$, i.e. the synthesis map, is a contraction and for 
\begin{equation*}
Sc=\int_{\Omega}c(\omega)u_\omega\,d\omega
\end{equation*}
the contraction property reads 
\begin{equation}\label{contraction}
    \|Sc\|^2_{L^2(X)}=\left\|\int_{\Omega}c(\omega)u_\omega\,d\omega\right\|_2^2\leq\int_{\Omega}|c(\omega)|^2\,d\omega.
\end{equation}
For our kernel convention we will also need the conjugate synthesis
\begin{equation*} %\label{conjugate synthesis}
    \overline{S}c:=\overline{S}(\overline{c})
\end{equation*}
Despite its notation, the two conjugations cancel and $\overline{S}$ is indeed a linear operator. Complex conjugation preserves Hilbert space norms and therefore the analogous inequality \eqref{contraction} holds also for the conjugated synthesis (spatial orthogonality inequality). In particular no individual Eisenstein series is assigned a global $L^2$ norm (see also Remark \ref{remark34}). We refer to \cite{mw2} for more details on the decomposition of the $L^2$-space $L^2(X)$.

\subsection{The automorphic kernel}

For $z,w\in G$ the automorphic kernel with respect to $f$ and $\Gamma$ is defined by
\begin{equation} \label{automorphickernel}
    K_{f}(z,w)=\sum_{\gamma\in\Gamma}f(z^{-1}\gamma w).
\end{equation}
The left $\Gamma$-invariance follows by re-indexing the sum and the right $K$-invariance follows from bi-$K$-invariance of $f$. Therefore this is a well defined function on the quotient $\Gamma\backslash G / K$. Plugging $f=k_{T}$ (defined by \eqref{ktdefinition}) in the automorphization \eqref{automorphickernel}, bi-$K$-invariance of the kernel $K_{k_T}$ follows from the invariance of the Frobenius norm. We simplify the notation $K_{T}(z,w):=K_{k_T}(z,w)$ and the sharp counting gives exactly 
\begin{equation}\label{auto kernel count}
    N_{n}(T;z,w)=K_{T}(z,w).
\end{equation}
To analyze the counting function spectrally we use the pre-trace formula.

\begin{proposition}[Pre-trace formula] \label{pretrace formula}

Let $f\in C_{c}^{\infty}(G)$ be a bi-$K$-invariant function. Then
\begin{equation}\label{kernel spectral}
K_{f}(z,w)=\int_{\Omega}h_f(\m_\omega)\overline{u_\omega (z)}u_\omega (w)\,d\omega.
\end{equation}
The contribution of the constant eigenfunction corresponding to $\m_0=\rho, \, u_0=V^{-1/2}$ is
\begin{equation*}
    \frac{1}{V}\int_{G}f(g)\,dg.
\end{equation*}
    For fixed $w$, the left hand side is a compactly supported, smooth function of $z$ on the quotient orbifold and in particular belongs in $L^2(X)$.
\end{proposition}

\begin{remark}
With our convention of the automorphic kernel, with the spectral expansion \eqref{kernel spectral} we will use either regular synthesis
    \begin{equation}\label{regular synthesis operator}
        Sc=\int_{\Omega}c(\omega)u_\omega \,d\omega,
    \end{equation}
    with
    \begin{equation*}
        c(\omega)=h_f(\m_\omega)\overline{u_\omega },
    \end{equation*}
    or conjugated synthesis
    \begin{equation}\label{conjugated synthesis operator}
        \overline{S}c=\int_{\Omega}c(\omega)\overline{u_\omega} \,d\omega, 
    \end{equation}
    with 
    \begin{equation*}
         c(\omega)=h_f(\m_\omega)u_\omega ,
    \end{equation*}
    depending on whether we fix $w$ and prove our assertion for almost all $z$ (in this case we use conjugated synthesis) or vice versa.
\end{remark}

\section{Bounds on the spectral parameters}

In this section we record two facts about the automorphic spectrum that will be used in the following sections. The first concerns a bound for the real part of every nonconstant spectral parameter, including parameters occurring in the continuous spectrum. The second is a local upper bound for spectral mass (local Weyl's law in shells).

\subsection{Non constant components}
The fact $\Re\m\in C_{\rho}$ alone gives $\b\leq (n-1)/2$, with equality at the constant $\m=\rho$. After removing the constant contribution we deduce the following stronger bound, needed for our counting argument. This is actually already known, see \cite[Theorem~1.1]{oh} and \cite[Theorem~1.5]{cou}, but we give a simpler and more concise proof here. 

    \begin{proposition}[Spectral gap]   \label{non constant bound}
    For $n\geq 3$, every non constant component of the spherical automorphic spectrum of $\slnz$ satisfies
    \begin{equation*}
       \b(\m_\omega) :=   \| \Re(\m_\omega)\|_{\infty} \leq \frac{n-2}{2}.
    \end{equation*}
This includes the cuspidal and residual discrete spectra and the induced continuous spectrum. 
\end{proposition}
\begin{proof}
For the parametrization of the spectrum and the deep input of the classification of the residual spectrum of general linear groups we follow \cite{blomerlutsko}; see also \cite{goldfeld, mw1, mw2}. We recall the parameters explicitly for completeness.
\par
Choose a partition of $n$ into $q$ blocks 
\begin{equation*}
    n=d_1+d_2+...+d_q,
\end{equation*}
where for each $j$ we choose a divisor $f_j \mid d_j$. For $f_j\geq 2$, choose a cuspidal automorphic datum on $GL_{f_j}$, whose archimidean parameter is denoted by $\nu_j=(\nu_{j,1},...,\nu_{j,f_j})$. The classical Jacquet-Shalika bounds stated in \cite[eq.~13]{blomerlutsko}, give
\begin{equation}\label{Jacquet Shalika}
  \|\Re\nu_j\|_{\infty}\leq \frac{1}{2}.
\end{equation}
For $f_j=1$, put $\nu_j=\textbf{0}$. Finally choose $s_j\in i\R$ satisfying 
$\sum_{j}d_js_j=0$ and put $\ell_j=d_j/f_j$. The entries contributing by the $j$-th block to the full parameter are equal to
\begin{equation}\label{entries}
    \nu_{j,a}+s_j+\frac{\ell_j-1}{2}-c,\quad 1\leq a\leq f_j,\,\,0\leq c\leq \ell_j-1.
\end{equation} 
In total, these blocks have $\sum_{j}f_j \, \ell_j=n$ entries. The cases $q=1$ include cuspidal representations and discrete Speh representations; varying the imaginary parameters for $q>1$ gives the induced continuous components. The formula \eqref{entries}, rather than a temperedness assumption, is what controls the possible large real parts.
\par
Since $s_j$ is imaginary, the maximum absolute real part in the $j$-th block is at most 
\begin{equation}\label{real part bound}
    \frac{1}{2}\left(\frac{d_j}{f_j}-1\right)+\frac{1}{2}\textbf{1}_{f_j\geq 2}.
\end{equation}
Indeed the shifts in \eqref{entries} range between $-(\ell_j-1)/2$ and $(\ell_j-1)/2$ and then \eqref{Jacquet Shalika} provides the extra factor $1/2$ when $f_j\geq 2$. This recovers exactly \cite[eq.~13]{blomerlutsko}.
\par 
Now if $f_j=1$, expression 
\eqref{real part bound} equals $(d_j-1)/2$. The case $d_j=n$ forces $q=1$; the inducing parameter is then zero and the block consists of the entries of $\rho$. This is precisely the constant representation in the parametrization under consideration. Once we remove this, we get $d_j\leq n-1$ and then 
\begin{equation*}
    \frac{d_j-1}{2}\leq\frac{n-2}{2}.
\end{equation*}
Moving on, if $f_j\geq 2$, then expression \eqref{real part bound} simplifies to $d_j/(2f_j)$, which is at most $n/4$.
\par
For $n=3$, the divisibility condition $f_j \mid d_j$ leaves only the choices $(d_j,f_j)=(2,2)$ or $(3,3)$ when $f_j\geq 2$. In either case $d_j/(2f_j)=1/2=(n-2)/2$. Finally, for $n\geq 4$ it is straightforward 
\begin{equation*}
    \frac{n}{4}\leq \frac{n-2}{2}
\end{equation*}
and the required bound follows. Therefore, every block satisfies the desired bound and taking the maximum over all the blocks proves the assertion over $\Omega'$.
\end{proof}

\begin{remark}
Our proof of Proposition \ref{non constant bound} uses only the Mœglin–Waldspurger classification of the spectrum, and does not depend on temperdness assumptions. It also follows as a corollary of Oh’s \cite{oh} explicit estimate, as mentioned above.
%not merely of Kazhdan's property \((T)\).
\end{remark}

\subsection{Local spectral mass}

Define the root-product weight 
\begin{equation*}
    \mathcal{W}(\t)=\prod_{1\leq i<j\leq n}(1+|\t_i-\t_j|),\quad \t\in\mathfrak{a}^\ast.
\end{equation*}
In this notation \cite[Lemma 1]{blomerlutsko} reads:
\begin{proposition}[Local spectral upper bound]\label{local spectral density}
    Let $\m\in i\mathfrak{a}^{\ast}$ , $B(\m)\subset \mathfrak{a}_{\C}^{\ast}$ a ball of size $O(1)$ about $\m$ and $w\in G$. Then 
\begin{equation*}
    \int_{B(\mu)}|u_\omega (w)|^2\,d\omega\ll_{w} \mathcal{W(\m)}.
\end{equation*}
\end{proposition}

For our proof we will need such spectral bounds over dyadic intervals. With $L\geq 1$  we define the shells
\begin{equation}\label{dyadic shells}
    \Omega_{L}=\{\omega\in \Omega:L\leq U(\m_\omega)\leq 2L\} \quad \text{and} \quad \Omega_{L}'=\{\omega\in \Omega':L\leq U(\m_\omega)\leq 2L\},
\end{equation}
where the latter consists of constituents with the constant removed and $U(\cdot)$ as in \eqref{spectral parameters}.

\begin{corollary}[A coarse shell bound]\label{blomer local weyl}
    Let $L\geq 1$ and $w\in G$ fixed. Then
\begin{equation}\label{coarse shell bound}
\int_{\Omega_L}|u_\omega(w)|^2\,d\omega\ll_{w,n}L^{d/2+r} .
\end{equation}
\end{corollary}
\begin{proof}
    The imaginary parameters lie in the $r$-dimensional trace-zero hyperplane. Its intersection with a box of side $O(L)$ is covered by $O_n(L^r)$ boxes of fixed side length. Choose an imaginary parameter $\t$ as center in each of these boxes. The real parts remain in the fixed compact set $C_\rho$, so enlarging the radius by a fixed amount covers the corresponding complex parameters by the balls in Proposition \ref{local spectral density}. On each of these balls $\mathcal{W}(\tau)\ll_n L^{d/2}$. Summing the local estimate over all $O_n(L^r)$ balls gives the desired bound. 
\end{proof}

\begin{remark} \label{remark34}
    The same reasoning justifies the use of a rapidly decaying spectral multiplier at a fixed point. If $a(\m)\ll_{A}(1+\|\m\|)^{-A}$, for every $A$ then 
    \begin{equation*}
        \int_{\Omega}|a(\m_\omega)|^2|u_\omega(w)|^2\,d\omega<\infty,
    \end{equation*}
    by summing \eqref{coarse shell bound} over dyadic $L$. Accordingly the synthesis against $\overline{u_{\omega}(z)}$ 
    is a well defined $L^2(X)$ function. The smooth kernels used in the proof below have such multipliers after the smoothing factor is inserted. This provides an explicit spectral justification alongside the geometric compact-support argument of Proposition \ref{pretrace formula}. 
\end{remark}

\section{The spherical transform of a Frobenius ball and smoothing}

We have seen in \eqref{auto kernel count} that the automorphic kernel corresponding to the characteristic function of the Frobenius ball equals exactly the counting function we are interested in. Due to smoothness reasons the characteristic function is not admissible in the pre-trace formula (Proposition \ref{pretrace formula}). As it is standard in these problems, before applying the pre-trace formula we convolute our original function with a smooth family of mollifiers, such that the convoluted functions approximate our original counting function as the mollifying parameter goes to $0$. 
\par
In our case, for $0<\delta<1$ we define $\psi_\delta:K\backslash G/K\rightarrow \R_{\geq 0}$ a smooth $L^1(AN)$-normalized function, supported in $B_\delta=\{g\in G:\max(\|g\|_{\text{op}},\|g^{-1}\|_{\text{op}})\leq 1+\delta\}$, a $\delta$-neighborhood of the identity element, where $\|\cdot\|_{\text{op}}$ is the operator norm. At this point we record facts about the spherical transform of the mollifier and its decay properties.

\begin{lemma}\cite[Lemma~2]{blomerlutsko} \label{mollifier decay properties}
Let $0<\delta<1$, any $A>0$ and $\psi_\delta$ as above. Also let 
\begin{equation*}
    \fpsi:=h_{\psi_\delta}, 
\end{equation*}
the spherical transform for this family of mollifiers. Then 
\begin{equation*}
\fpsi(\rho)=1 ,\quad  |\fpsi(\m)|\ll_{n,A}(1+\delta U)^{-A}.
\end{equation*}
\end{lemma}
Define the convolution 
\begin{equation*}
k_{T,\delta}:=k_T \ast \psi_\delta.
\end{equation*}
Theorem \ref{ourmainthm} for the original counting function will follow from bounds for the functions $k_{T(1+\delta)^{\pm1},\delta}$ and the double inequality 
\begin{equation}\label{sandwich}
    k_{T(1+\delta)^{-1},\delta}\leq k_{T}\leq k_{T(1+\delta),\delta},
\end{equation}
Let
\begin{eqnarray}\label{pm notations}
N_{\delta}^{\pm}(T;z,w)=\sum_{\gamma \in \Gamma}(k_{T(1+\delta)^{\pm 1},\delta})(z^{-1}\gamma w),
\end{eqnarray}
then the periodization of \eqref{sandwich} implies 
\begin{equation}\label{sandwich 2}
    N_{\delta}^{-}(T;z,w)\leq N_n(T;z,w)\leq N^{+}_\delta(T;z,w).
\end{equation}
Let us denote the spherical transforms of the approximations by
\begin{equation*}
h_{T_\pm}=h_{k_{T(1+\delta)^{\pm1}}}.
\end{equation*}
The multiplicativity of the spherical transform with respect to the convolution \eqref{convolution spherical} implies that 
\begin{equation*}
    h_{k_{T(1+\delta)^{\pm1}}\ast \psi_\delta}=h_{T_\pm}\cdot \fpsi.
\end{equation*}
With this notation the contribution of the constant parameter $\m=\rho$ to $N_{\delta}^{\pm}(T;z,w)$ is exactly 
\begin{equation}\label{smoothed constant}
    M_0(T(1+\delta)^{\pm 1}),
\end{equation}
with $M_0$ defined as in \eqref{m0}. This follows from the fact that $\fpsi(\rho)=1$, the definition of $N_{\delta}^{\pm}$ and Proposition \ref{pretrace formula}.
\par
Essential for our proof will be an estimate for the spherical transform of a Frobenius ball similar to  \cite[Lemma~3]{blomerlutsko}. However, a key difference is that we also provide an estimate for the derivative of the spherical transform, with respect to the radial parameter $T$. This will help us obtain better decay properties in the form of a unit-interval Sobolev inequality.
\begin{proposition}\label{new transform bound}
    Let $n\geq 3$, $\eta>0$, $\k>0$ and $C\geq 1$. Let $h_T=h_{k_{T}}$ and define the operator $D=T\partial_T$. Uniformly for $T \geq 2$, for parameters $\m$ in the spherical unitary spectrum and for $U\leq C T^{2-\k}$ we get 
    \begin{equation}\label{Dq estimate}
|D^{q}h_T(\mu)|\ll_{n,\eta,\kappa}T^{d/2+n\beta+\eta}U^{q-d/4-(1+\beta)/2}\mathcal{G}(\t)   \end{equation}
for $q=0,1$, where $\beta = \beta(\mu)$ and 
    \begin{equation}\label{G definition}
        \mathcal{G}(\t)=\prod_{j=1}^{n}(1+\max_i \t_i-\t_j)^{-1/2}+\prod_{j=1}^{n}(1+\t_j-\min_i \t_i)^{-1/2}.
    \end{equation}
\end{proposition}
\begin{proof}
    We begin with the Mellin representation of the spherical transform 
    \begin{equation}\label{spherical}
    h_T(\m)= \frac{\pi^{d/4}}{2^n}T^{d/2}\int_{(c)}T^{ns}\frac{\prod_{j=1}^{n}\Gamma\left(\frac{s-\m_j}{2}\right)}{\Gamma\left(ns/2+d/4+1\right)}\,\frac{ds}{2\pi i},
\end{equation}
for $c>0$ sufficiently large. This can be proven by the definition of the spherical transform employing the Haar measures for the subgroups $A$ and $N$ (see \cite[Lemma ~3]{blomerlutsko}). For the vertical line of integration $(c)$ we write $c=s+it$ and for the coordinates of the spectral parameter $\m_j=m_j+i\t_j$, where without loss of generality (due to Weyl invariance) we have ordered the imaginary parts such that $\t_1\geq...\geq \t_n$. Consequently we get $\sum\t_j=0$, $\sum m_j=0$ and $\t_1\asymp-\t_n\asymp\|\t\|_\infty$. Stirling's formula reads 
\begin{equation*}
    |\Gamma(x+it)|\sim\sqrt{2\pi}e^{-\pi|t|/2}|t|^{x-1/2}.
\end{equation*}
We will denote with $pol$ the exponential behavior and $exp$ the exponential behavior of the Gamma quotient
\begin{equation*}
    \frac{\prod_{j=1}^{n}\Gamma\left(\frac{s-\m_j}{2}\right)}{\Gamma\left(ns/2+d/4+1\right)}.
\end{equation*}
Then we get
\begin{eqnarray*}
&&pol\ll_{\sigma}\left(\frac{1}{1+|t|}\right)^{\frac{n\sigma}{2}+\frac{d}{4}+\frac{1}{2}}\prod_{j=1}^{n}(1+|t-\t_j|)^{\frac{\s-m_j-1}{2}},\\
&&exp= e^{-\frac{\pi}{4}\mathcal{E}(t)},\quad \mathcal{E}(t)=\sum_{j=1}^{n}|t-\t_j|-n|t|.
\end{eqnarray*}
On a vertical line $(c)$ away from the poles, the far tail of the Gamma quotient is bounded by 
\begin{equation*}
    O_{n,c}(\sigma^{-\frac{n^2+n+2}{4}}).
\end{equation*}
The same tail after multiplication by $d/2+ns$ is integrable for every $n\geq 3$. Dominated convergence proves continuous differentiability and allows the passage of the $D$ operator inside the integral. This also provides sufficient decay to allow us to move freely the line of integration. We get 
\begin{equation*}
    D^qh_T(\m)=\frac{\pi^{d/4}}{2^n}T^{d/2}\int_{(c)}(d/2+ns)^{q}T^{ns}\frac{\prod_{j=1}^{n}\Gamma\left(\frac{s-\m_j}{2}\right)}{\Gamma\left(ns/2+d/4+1\right)}\,\frac{ds}{2\pi i},
\end{equation*}
and using $\Gamma(z+1)=z\Gamma(z)$ this becomes
\begin{equation}\label{Dq integral}
    D^qh_T(\m)=\frac{\pi^{d/4}}{2^{n-q}}T^{d/2}\int_{(c)}T^{ns}\frac{\prod_{j=1}^{n}\Gamma\left(\frac{s-\m_j}{2}\right)}{\Gamma\left(ns/2+d/4+1-q\right)}\,\frac{ds}{2\pi i}.
\end{equation}
For $q=0$ the assertion is proven in \cite[Lemma ~3]{blomerlutsko}, so for completeness we need to include the case $q=1$. We can also assume that $0<\k<2$, since a larger value restricts the range further and bounded admissible values of $T$ can be absorbed in the constant. Let
\begin{equation*}
    Q_{q}(s,\m)=\frac{\prod_j \Gamma((s-\m_j)/2)}{\Gamma(ns/2+d/4+1-q)}.
\end{equation*}
The polynomial and exponential behavior of $Q_q(s,\m)$ follow in a similar fashion as above, where for $q=1$ we simply get an additional factor $-q$ in the exponent of the first term of $pol$, i.e.
\begin{equation}\label{gamma quotient bounds}
    |Q_{q}(s,\m)|\ll_\s \left(\frac{1}{1+|t|}\right)^{\frac{n\sigma}{2}+\frac{d}{4}+\frac{1}{2}-q}\prod_{j=1}^{n}(1+|t-\t_j|)^{\frac{\s-m_j-1}{2}}\cdot e^{-\frac{\pi}{4}\mathcal{E}(t)}
\end{equation}
The possible poles of $Q_{q}(s,\m)$ are 
\begin{equation}\label{poles}
    p_{j,k}=\m_j-2k,\quad 1\leq j\leq n,\quad k=0,1,... \quad .
\end{equation}
Some might be canceled by zeros of the reciprocal denominator. It is convenient to retain them in the list since counting a canceled pole causes no difficulty.
\\\\
\textbf{The contour to the far left.} Choose a large number $L$, with $L    \geq n+2$, so that the line $\Re s=-L$ has a fixed positive distance, depending only on $n$, from every $\Re p_{j,k} =m_j-2k$ of a numerator pole. Such a choice can be made in any prescribed unit interval sufficiently far to the left; for each $j$, the numbers $m_j-2k$ have spacing $2$, so only $O_n(1)$ of them are relevant to that interval. We shall increase the prescribed interval and hence $L$, when a larger saving is needed.
\par
On $s=-L+it$, with $|t|\leq 2U$ and $L$ large enough, all numerator exponents in \eqref{gamma quotient bounds} are negative (recall that $\Re\m\in C_\rho$). Dropping their factors and simply bounding the exponential contribution by $O(1)$, we get for a fixed constant $B_n$,
\begin{equation*}
    |Q_q(-L+it,\m)|\ll_{n,L}U^{nL/2 +B_n}.
\end{equation*}
The contribution of this segment to the integral \eqref{Dq integral} is 
\begin{equation}\label{far left 1}
    O_{n,L}(T^{d/2-nL} U^{nL/2+B_n})=O_{n,L,C}(T^{d/2-n\k L/2 +2B_n}).
\end{equation}
For $|t|>2U$ every $1+|t-\t_j|$ appearing in the product in \eqref{gamma quotient bounds} is $\asymp$ to $1+|t|$. Working out the exponents in \eqref{gamma quotient bounds} the contribution of the tail to the integral is
\begin{equation}\label{far left 2}
    O_{n,L}(T^{d/2-nL} U^{1-\frac{d}{4}-\frac{n+1}{2}+q})=O_{n,L}(T^{d/2-nL}U^{1-A_n+q})
\end{equation}
with $A_n=d/4+(n+1)/2> 2$, for every $ n\geq 3$.
\par
It follows that the whole integral on the far left line is $O(T^{-B})$ for any prescribed $B>0$, upon choosing $L$ sufficiently large in terms of $n,\k , B$. The constants may depend on $L$, which is now fixed. 
\\\\
\textbf{Grouping the poles.} Between the initial vertical line $(c) = \{\Re s = c\}$ and the vertical line $(-L) = \{\Re s=-L\}$ there are at most $M=M(n,L)$ poles $p_{j,k}$, counting multiplicity. This number is independent of $U$. Merely requiring the $\m_j$ to be distinct would not make the poles distinct: $\m_j-\m_i\in 2\Z$ can also produce coincidences. Moreover, estimates of  separate simple residues are not uniform as such a coincidence is approached.
\par
Choose a small number $a$. Around each crossed point $p_{j,k}$, draw a disc of radius $a$ and take the boundaries of the connected components of their union (see figure \ref{fig:poles pic}). The radius is chosen smaller than the fixed separation from both vertical lines. Every boundary point has distance at least $a$ from each crossed pole and has a fixed positive distance from every uncrossed pole. The total boundary length is at most $2\pi aM$, and the diameter of each component is at most $2aM$. The boundaries may be treated as piecewise smooth contours; an arbitrary small adjustment avoids any tangencies. The residue theorem expresses the sum of residues inside a component as the integral on its boundary, with no separation assumption on the individual poles.

Select a point $p_{j,k}=m_j-2k+i\t_j$ in a given component and parametrize $s=\sigma+it$ the points on its boundary. Then
\begin{equation*}
    \sigma=m_j-2k+O(Ma),\quad t=\t_j+O(Ma).
\end{equation*}
Set 
\begin{equation*}
    U_j=1+|\t_j|,\quad d_{ij}=1+|\t_j-\t_i|,\quad d_j=\min(\t_1-\t_j,\t_j-\t_n),
\end{equation*}
where the imaginary parts $\tau_j$ have been ordered. 
\par
\begin{figure}[H]
\centering
\begin{tikzpicture}
[
    scale=1.0,
    redboundary/.style={draw=red!80!black, line width=1.4pt},
    vline/.style={draw=blue!60!black, thick, dashed},
    uncrossed/.style={circle, fill=black, inner sep=0pt, minimum size=2.2pt},
    poledot/.style={circle, fill=black, inner sep=0pt, minimum size=2.4pt},
    label/.style={font=\small}
]

\def\aa{0.42}
\def\xL{-2.6}        % left vertical line  s = -L + it
\def\xR{3.4}         % right vertical line s = c + it
\def\yLo{-2.4}
\def\yHi{2.6}

% vertical lines
\draw[vline] (\xL,\yLo) -- (\xL,\yHi);
\node[label,blue!60!black] at (\xL,\yHi+0.18) {$s=-L+it$};
\draw[vline] (\xR,\yLo) -- (\xR,\yHi);
\node[label,blue!60!black] at (\xR,\yHi+0.18) {$s=c+it$};

% pole centres
\coordinate (P1) at (-1.05, 0.95);
\coordinate (P2) at (-0.62, 0.62);
\coordinate (P3) at (-0.90, 0.28);
\coordinate (P4) at (1.35, 1.55);
\coordinate (P5) at (1.82, 1.35);
\coordinate (P6) at (1.55, 0.95);
\coordinate (P7) at (0.85,-1.45);
\coordinate (P8) at (2.45,-0.95);
\coordinate (P9) at (2.85,-1.25);

% ---------- COMPONENT 1 : P1,P2,P3 ----------
\begin{scope}
  \fill[gray!15] (P1) circle (\aa);
  \fill[gray!15] (P2) circle (\aa);
  \fill[gray!15] (P3) circle (\aa);
\end{scope}
\draw[redboundary] (P1) circle (\aa);
\draw[redboundary] (P2) circle (\aa);
\draw[redboundary] (P3) circle (\aa);
\begin{scope}
  \fill[gray!15] (P1) circle (\aa);
  \fill[gray!15] (P2) circle (\aa);
  \fill[gray!15] (P3) circle (\aa);
\end{scope}
\draw[gray!55, line width=0.6pt] (P1) circle (\aa);
\draw[gray!55, line width=0.6pt] (P2) circle (\aa);
\draw[gray!55, line width=0.6pt] (P3) circle (\aa);
\node[poledot] at (P1) {};
\node[poledot] at (P2) {};
\node[poledot] at (P3) {};

% ---------- COMPONENT 2 : P4,P5,P6 ----------
\begin{scope}
  \fill[gray!15] (P4) circle (\aa);
  \fill[gray!15] (P5) circle (\aa);
  \fill[gray!15] (P6) circle (\aa);
\end{scope}
\draw[redboundary] (P4) circle (\aa);
\draw[redboundary] (P5) circle (\aa);
\draw[redboundary] (P6) circle (\aa);
\begin{scope}
  \fill[gray!15] (P4) circle (\aa);
  \fill[gray!15] (P5) circle (\aa);
  \fill[gray!15] (P6) circle (\aa);
\end{scope}
\draw[gray!55, line width=0.6pt] (P4) circle (\aa);
\draw[gray!55, line width=0.6pt] (P5) circle (\aa);
\draw[gray!55, line width=0.6pt] (P6) circle (\aa);
\node[poledot] at (P4) {};
\node[poledot] at (P5) {};
\node[poledot] at (P6) {};

% ---------- COMPONENT 3 : P7 ----------
\fill[gray!15] (P7) circle (\aa);
\draw[redboundary] (P7) circle (\aa);
\node[poledot] at (P7) {};

% ---------- COMPONENT 4 : P8,P9 ----------
\begin{scope}
  \fill[gray!15] (P8) circle (\aa);
  \fill[gray!15] (P9) circle (\aa);
\end{scope}
\draw[redboundary] (P8) circle (\aa);
\draw[redboundary] (P9) circle (\aa);
\begin{scope}
  \fill[gray!15] (P8) circle (\aa);
  \fill[gray!15] (P9) circle (\aa);
\end{scope}
\draw[gray!55, line width=0.6pt] (P8) circle (\aa);
\draw[gray!55, line width=0.6pt] (P9) circle (\aa);
\node[poledot] at (P8) {};
\node[poledot] at (P9) {};

% radius annotation
\draw[<->,thick] (P7) -- ($(P7)+(\aa,0)$);
\node[label] at ($(P7)+(0.30,-0.22)$) {$a$};

% separation annotation
\draw[<->,thick,blue!60!black] (\xL,\yLo-0.35) -- ($(-1.42,\yLo-0.35)$);
\node[label,blue!60!black] at ($(\xL-0.18,\yLo-0.62)$) {$ \ \ \ \ \ \ \ \ \ \ \ \ \ >a$};

% diameter annotation
\draw[<->,thick,red!75!black] (1.00,2.35) -- (2.22,2.35);
\node[label,red!75!black] at (1.61,2.55) {$\leq 2aM$};

\end{tikzpicture}
\caption{Grouping the poles.}
    \label{fig:poles pic}
\end{figure}
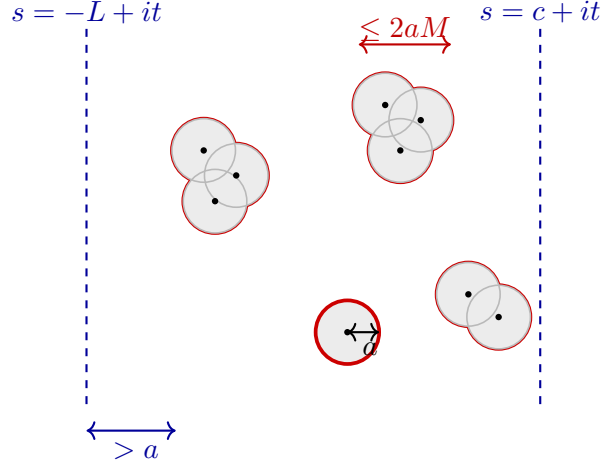

Using the bounds in \eqref{gamma quotient bounds} the integral in \eqref{Dq integral} along the component boundary, before multiplying by factor the $T^{d/2}$, is bounded by
\begin{equation}\label{component bound}
C_{n,L,a} T^{C_n Ma}U^{C_n Ma}R_{j,k,q},
\end{equation}
where
\begin{equation*}%\label{Rjkq}
R_{j,k,q}=T^{nm_j-nk} U_j^{-nm_j/2+nk-d/4-1/2+q} e^{-\pi d_j/2}\prod_{i=1}^{n}d_{ij}^{(m_j-m_i-1)/2-k}.
\end{equation*}
\par 
We spell out the loss in \eqref{component bound}. Perturbating $\sigma$ by $O(Ma)$ changes $T^{n\sigma}$ by at most $T^{C_n Ma}$. It changes the exponents of the polynomials in \eqref{gamma quotient bounds} by $O_n(Ma)$, producing at most $U^{C_n Ma}$. Bounded perturbations of the imaginary argument change $1+|t|$ and $1+|t-\t_i|$ by bounded ratios. Finally for the exponential, we use $\sum_{j=1}^{n}|t-\t_j|-n|t|\geq 2\min(|\t_1-t|,|t-\t_n|),$ for $\t_n\leq t\leq \t_1$, and $0$ outside this region and also the fact that $\mathcal{E}(t)$ is Lipschitz   with constant at most $2n$ so its exponential changes only by a bounded factor. These observations prove \eqref{component bound}, including at multiple poles.
\\\\
\textbf{Reduction to $R_{j,0,q}$.} The ratio between $R_{j,k,q}$ and $R_{j,0,q}$ is 
\begin{equation*}
   \left(\frac{U_j}{T}\right)^{nk}\prod_{i=1}^{n}d_{ij}^{-k}.
\end{equation*}
Since $U_j\leq U\leq C T^{2-\k}$, this ratio is bounded by a constant depending only on $n,L,C$. Therefore it suffices to estimate $R_{j,0,q}$. This argument is valid even if a pole that formally appears in \eqref{poles} is canceled: we are only bounding the corresponding contour integral.
\\\\
\textbf{Replacement of the real part by its maximum.} In the explicit expression of $R_{j,0,q}$, we form $R'_{j,0,k}$, by replacing the occurrences of $m_j$ outside the list $m_1,...,m_n$ by $\b=\|\Re \mu\|_\infty$, while keeping each $m_i$ in that list fixed, i.e.
\begin{equation}\label{R'}
    R'_{j,0,q}=T^{n\beta} U_j^{-n\beta/2-d/4-1/2+q} 
    e^{-\pi d_j/2}\prod_{i=1}^{n}d_{ij}^{(\beta-m_i-1)/2}.
\end{equation}
The ratio of the original expression to the new one is 
\begin{equation*}
\frac{R_{j,0,q}}{R'_{j,0,q}}=T^{n(m_j-\b)}U_j^{-n(m_j-\b)/2}\prod_{i=1}^{n}d_{ij}^{\frac{m_j-\b}{2}}\leq \left(\frac{U_j}{T^2}\right)^{n(\b-m_j)/2}\ll_{n,C} 1,
\end{equation*}
where we used that $m_j\leq \b$, hence every factor in the product is at most $1$ and then used $U_j\leq C T^{2-\k}$. This provides the precise comparison behind the monotonicity in the real part used in \cite[p.10]{blomerlutsko}.
\\\\
\textbf{Removal of the interior imaginary parts.} If $U$ is bounded, all the remaining spectral factors are bounded above and below by constants depending only on $n$ and there is nothing further to prove. Suppose that $U$ is large. Since $\t_1\asymp -\t_n\asymp \|\t\|_\infty\asymp_{n} U$, then either $U\asymp U_j$, or $d_j\geq cU $ for a fixed $c>0$ depending only on $n$. In the latter case $e^{-\pi d_j /2}$ absorbs every fixed power of $U$. All other spectral factors have exponents bounded in terms of $n$; moreover, with $\mathcal{G}(\t)$ as in \eqref{G definition} 
\begin{equation*}%\label{G lower}
\mathcal{G}(\t)\gg_n U^{-(n-1)/2}.
\end{equation*}
Indeed, each factor in the products of $\eqref{G definition}$, without the power, are smaller or equal than a quantity which is $\asymp U$, the negative power switches the inequality and there are $n-1$ factors in the product.
 \par
 Each term in the product of \eqref{R'} is $\ll U$ and the exponents are bounded by $(\beta-1)/2$. There are $n-1$ factors in the sum contributing, since the $i=j$ term is equal to $1$, therefore the whole product is $\ll_n U^{(n-1)(\beta-1)/2}$ and in this case
 \begin{equation*}
R'_{j,0,q}\ll T^{n\beta}U^{q-d/4-(1+\beta)/2}\mathcal{G}(\t),
 \end{equation*}
satisfying the desired estimate.
We may henceforth assume $U_j\asymp_n U$.
\\\\
\textbf{The use of unitarity.} After replacement by $\beta$ all numbers $\beta-m_i$ are non negative and 
\begin{equation*}
    \sum_{i=1}^{n}\frac{\b-m_i}{2}=\frac{n\b }{2}.
\end{equation*}
There is an index $i_0$ such that 
\begin{equation}\label{i_0 imaginary}
\t_{i_0}=\t_j, \quad m_{i_0}\leq 0.
\end{equation}
Indeed, if $m_j\leq 0$ take $i_0=j.$ If $m_j>0$, the unitary symmetry  supplies the coordinate $-\overline{\m_j}=-m_j+i\t_j$. For this index $d_{i_0 ,j}=1$; hence the exponent $(\b-m_{i_0})/2$, which is at least $\beta/2$, contributes no power of $U$. All remaining $n-1$ factors satisfy $d_{ij}=O_n(U)$. It follows that 
\begin{equation}\label{endpoints 1}
\prod_{i=1}^{n}d_{ij}^{\frac{\beta-m_i}{2}}\ll_n U^{n\beta/2-\beta/2}. 
\end{equation}
Dividing this estimate by $U_j^{n\beta/2}$ leaves precisely the factor $U^{-\beta/2}$. This argument explains why the unitary symmetry, and not just a bounded real strip, is needed for the stated sharp estimate. It also proves the estimate when some or all imaginary coordinates coincide.
\\\\
\textbf{Reduction to the endpoint products.} Choose the closest endpoint $\psi\in\{\t_1,\t_n\}$ to $\t_j$. Then $|\psi-\t_j|=d_j$, and for every $i$ 
\begin{equation*}
    1+|\psi-\t_i|\leq(1+d_j)(1+|\t_j-\t_i|).
\end{equation*}
Thus 
\begin{equation}\label{endpoints 2}
    e^{-\pi d_j/2}\prod_{i=1}^{n}d_{ij}^{-1/2}\leq e^{-\pi d_j/2}(1+d_j)^{n/2}\prod_{i=1}^{n}(1+|\psi-\t_i|)^{-1/2}\ll_{n}\mathcal{G}(\t),
\end{equation}
where the exponential kills the polynomial growth in $d_j$ and the remaining product is either one of the products of $\mathcal{G}(\tau)$ depending on which value $\psi$ takes.
Combining \eqref{endpoints 1} and \eqref{endpoints 2} and recalling $U\asymp_n U_j$, we get
\begin{equation}\label{endpoints 3}
    R_{j,k,q}\ll_{n,L,C}T^{n\beta} U^{q-d/4-(1+\beta)/2} \mathcal{G}(\t).
\end{equation}
\\
\textbf{Completion of the contour estimate.} There are at most $M=M(n,L)$ components. Fix their radius $a$ so small so that the product $T^{C_n Ma}U^{C_n Ma}$ is $O(T^{\eta})$, for some $\eta>0$ and small. This is possible, since $U\ll T^2$. Multiply \eqref{endpoints 3} by the outside factor $T^{d/2}$ and sum over the components. The contribution of the far left line can be made smaller than any fixed negative power of $T$, by \eqref{far left 1} and \eqref{far left 2}. It can therefore be absorbed in \eqref{Dq estimate}: using $\b\leq (n-1)/2$,  and $U\ll T^2$ the right hand side of \eqref{Dq estimate} is bounded below by some negative power of $T$. This completes the proof.
\end{proof}

\begin{remark}
    The contour argument distinguishes two different uniformity issues. The discs control coincidences of the complex poles $\m_j-2k $ and the endpoint products \eqref{i_0 imaginary} control coincidences of the imaginary coordinates. Both are required for a statement uniform across the full automorphic spectrum. At an exact multiple pole, direct residue formulas contain powers of $\log T$. The small contour argument handles these and nearby poles at once; the resulting loss is the arbitrarily small factor $T^{\eta}$. 
\end{remark}

\subsection{The exact main term, its asymptotic and coarse bounds}

Let $V$ denote the volume used in the spectral normalization, and
\begin{equation*}
    M_0(T)=\frac{\hbox{vol}(B_T)}{V},\quad M_n(T)=c_nT^d,
\end{equation*}
as in \eqref{setting}. The distinction between $M_0$ and $M_n$ is useful: the constant spectral coefficient is the exact ball volume, whereas the stated asymptotic formula uses its leading term. 
\begin{lemma}[Approximation of the main term]\label{ball estimates}
For every $\eta>0$, as $T\to\infty$,
\begin{equation}  \label{volume of ball bound}
\hbox{vol}(B_T) = a_nT^d+O_{n,\eta}(T^{d-n+\eta}),  \end{equation}
and 
\begin{equation} \label{volume of D ball bound}
D\hbox{vol}(B_T) = da_n \, T^d+O_{n,\eta}(T^{d-n+\eta}),
\end{equation}
where
\begin{equation} \label{leading term coefficient}
a_n = \frac{\pi^{\frac{d}{4}}}{2^{n-1}\Gamma(d/2+1)}\prod_{j=1}^{n-1}\Gamma(j/2).
\end{equation}
In particular
\begin{equation}\label{main term asymptotics}
    M_0(T)=M_n(T)+O_{n,\eta}(T^{d-n+\eta})
\end{equation}
and
\begin{equation} \label{main term asymptotics2}
D \, M_0(T)= dc_nT^d+O_{n,\eta}(T^{d-n+\eta}),
\end{equation}
with $c_n=a_n/V$. Moreover, for $0<\delta\leq1/2$ and sufficiently large $T$,
\begin{equation}\label{main term approximation}
|M_0(T(1+\delta)^{\pm 1})-M_0 (T)|\ll_n\delta T^d.
\end{equation}
\end{lemma}
\begin{proof}
    The constant spherical function satisfies $\phi_\rho (g)=1$. Using the spherical transform  \eqref{spherical 1 step} and $h_T$ as in Proposition \ref{new transform bound}  then $h_{T}(\rho)=\hbox{vol}(B_T)$. Applying \eqref{spherical} with $\m=\rho$ we get 
    \begin{equation}\label{mellin for constant term}
        h_T(\rho)=\frac{\pi^{\frac{d}{4}}}{2^{n}}T^{d/2}\int_{(c)}T^{ns}\frac{\prod_{j=1}^n\Gamma((s-\rho_j)/2)}{\Gamma(ns/2+d/4+1)}\,\frac{ds}{2\pi i},
    \end{equation}
    with $\rho_j=\frac{n-2j+1}{2}$.
    The possible poles are 
    \begin{equation*}
        p_{j,k}=\rho_j-2k,
    \end{equation*}
    and the rightmost is $s_0=(n-1)/2$.
    This is a simple pole, since only the first numerator factor has a pole there and the denominator is nonzero. The residue of $\Gamma((s-s_0)/2)$ is $2$. For the other factors,
    \begin{equation*}
        \frac{s_0-\rho_j}{2}=\frac{j-1}{2},\, (2\leq j\leq n),\quad \frac{ns_0}{2}+\frac{d}{4}+1=\frac{d}{2}+1.
    \end{equation*}
    The outside and inside powers of $T$ combine to give $T^{d/2+ns_0}=T^{d/2+n(n-1)/2}=T^d$. Hence the residue in \eqref{mellin for constant term}  is exactly $a_n T^d$, with $a_n$ as in \eqref{leading term coefficient}.
    The next possible numerator pole is at $s_0-1=(n-3)/2$. Choose a line just to its right, say at $\Re s=s_0-1+2\eta/n$; if necessary first replace $\eta$ by a smaller positive number to keep this line below $s_0$. The contour shift crosses only the rightmost pole and the new integral is absolutely convergent by \eqref{gamma quotient bounds}. Its power of $T$ is 
    \begin{equation*}
        d/2+n(s_0-1+2\eta/n)=d-n+\eta.    \end{equation*}
    This proves \eqref{volume of ball bound}. Apply the same argument to \eqref{Dq integral} with $q=1$. Its denominator at the leading pole is $\Gamma(d/2)$ instead of $\Gamma(d/2+1)$ multiplying the residue by $d/2$. The remaining integral has the same power of $T$ proving \eqref{volume of D ball bound}. Dividing the two asymptotics by $V$ proves \eqref{main term asymptotics}. Finally \eqref{volume of D ball bound} gives $DM_0(Y)\ll_{n} Y^{d}$ for large $Y$. Integrating with respect to $dY/Y$ between $T$ and $T(1+\delta)^{\pm1}$ proves \eqref{main term approximation}. 
\end{proof}

\begin{lemma}[Coarse bounds]  \label{coarse bounds}
Uniformly for parameters satisfying \eqref{spectral parameter conditions} and for sufficiently large $T$
\begin{equation*}
    |h_T(\m)|\leq \hbox{vol}(B_T)\ll_n T^d,\quad |Dh_T(\m)|\leq D\hbox{vol}(B_T)\ll_n T^d.
\end{equation*}
The estimates have no restrictions on $U$.
    
\end{lemma}
\begin{proof}
  On the unitary spherical spectrum we have $|\phi_{-\m}(g)|\leq 1$. Integration over the ball proves the first inequality. For $Y>T$ the nesting of the balls gives 
  \begin{equation*}
      |h_Y(\m)-h_T(\m)|\leq \int_{B_Y\backslash B_T}|\phi_{-\m}(g)|\,dg\leq \hbox{vol}(B_Y)-\hbox{vol}(B_T).
  \end{equation*}
  The same inequality, with the endpoints reversed holds for $Y<T$. Continuous differentiability as function of the radius was established through dominated convergence in the proof of Proposition \ref{new transform bound}. Dividing by $|Y-T|$ and passing to the limit gives $|\partial_Th_T(\m)|\leq \partial _T\hbox{vol}(B_T)$. Multiplication by $T$ and Lemma \ref{ball estimates} proves the second estimate. In particular, no separate smooth-boundary assertion for the Frobenius ball is needed. 
\end{proof}

\section{Root products and spectral estimates}

The transform estimate contains an anisotropic factor. Its interaction with the local spectral density is the reason a crude dimension count would give the wrong answer. We record this interaction before estimating any spectral sums.
\\\\
Recall 
\begin{equation*}
    \mathcal{W}(\t)=\prod_{1\leq i<j\leq n}(1+|\t_i-\t_j|),\quad U=1+\|\t\|_\infty .
\end{equation*}

The factor $\mathcal{G}(\t)$ in \eqref{G definition} is the sum of the two endpoint products. As before we order the imaginary parts $\t_1\geq ...\geq \t_n$ and put
 \begin{equation*}
        P_{+}=\prod_{j=2}^{n}(1+\t_1-\t_j),\quad P_{-}=\prod_{j=1}^{n-1}(1+\t_j-\t_n).
    \end{equation*}
Then it is straightforward that $\mathcal{G}=P_{+}^{-1/2}+P_{-}^{-1/2}$.

\begin{lemma}\label{spectral cancellation}
   Uniformly for all real $\t$, with $\sum_{j}\t_j=0$, including the walls of Weyl chambers, we have the bound
    \begin{equation}\label{cancellation in spectral}
        \mathcal{G}^2(\tau) \mathcal{W}(\tau)\ll_n U^{d/2-r}.
    \end{equation}
\end{lemma}
\begin{proof}
    The elementary inequality $(a+b)^2\leq 2a^2+2b^2$ gives
    \begin{equation*}
        \mathcal{G}^2\mathcal{W}\leq 2\frac{\mathcal{W}}{P_+}+2\frac{\mathcal{W}}{P_-}.
    \end{equation*}
    Each quotient cancels exactly $r=n-1$ of the $d/2$ factors in $\mathcal{W}$. Every remaining factor is at most $2U$ and there are $d/2-r$ of them. This proves 
\eqref{cancellation in spectral}. The use of $1+|\t_i-\t_j|$ rather than a homogeneous power of the differences is what makes the proof valid at collisions.
\end{proof}

For example, when $n=3$ write $A=1+\t_1-\t_2$, $B=1+\t_2-\t_3$ and $C=1+\t_1-\t_3$. Then
\begin{equation*}
\mathcal{W}= ABC,\quad \mathcal{G}=(AC)^{-1/2}+(BC)^{-1/2},\quad \mathcal{G}^2\mathcal{W}=(\sqrt{A}+\sqrt{B})^2\ll U.
\end{equation*}
This concrete identity illustrates the cancellation without any assumption that the parameters are in the interior of a chamber.
\\\\
Let dyadic $L\in\{1,2,4,...\}$ and recall the full non-constant spectral packet
\begin{equation*}
    \Omega_{L}'=\{\omega\in\Omega':L\leq U(\mu_\omega)\leq 2L\}.
\end{equation*}
Here $\Omega'$ is the full spectral parameter space with the constant representation removed: it includes every continuous spectrum measure and every multiplicity label. These packets are disjoint and cover $\Omega'$.

In this subsection we exploit the cancellation in Lemma \ref{cancellation in spectral} to obtain a bound for 
\begin{equation*}
    \int_{\Omega_L'}\mathcal{G}^2(\t_\omega)|u_{\omega}(w)|^2\,d\omega.
\end{equation*}
As mentioned in the proof of Corollary \ref{blomer local weyl} we cover a shell by $O_n(L^r)$ balls of fixed radius about purely imaginary parameters.
On a ball with fixed radius, the quantities $U$ and $1+|\t_i-\t_j|$ are comparable to their values at the center, with constants depending on the radius and on $n$. As for the endpoint products, we use the fact that $\max$ and $\min$ functions are Lipschitz.  Applying the local spectral estimate of Proposition \ref{local spectral density} and then Lemma \ref{spectral cancellation} we obtain:
\begin{eqnarray}\label{first spectral mass on shells}
\int_{\Omega_L'}\mathcal{G}^2(\t_\omega)|u_{\omega}(w)|^2\,d\omega &=& \sum_{balls}\int_{B(\t_\omega)}G^2(\t_{\omega)}|u_{\omega}(w)|^2\,d\omega \nonumber \\ 
&\ll_n& \sum_{balls}\|G^2(\t_\omega)\|_\infty \int_{B(\tau_\omega)}|u_{\omega}(w)|^2\,d\omega \\ 
&\ll_{n,w}& \sum_{balls}\|G^2(\t)\mathcal{W}(\tau)\|_\infty\ll_{n,w}L^r U^{d/2-r}\ll _{n,w}L^{d/2}.  \nonumber
\end{eqnarray}

\section{A maximal estimate on a multiplicative radius interval}

Fix $w\in X$ for the remainder of the proof. Every implicit constant in this section may depend on this point. Choose once and for all 
\begin{equation*}
0<a<2, \quad Q\geq Q_0,\quad\delta=Q^{-a},\quad \k=\frac{2-a}{2},\quad L_0=Q^{2-\k}.
\end{equation*}
Parametrize the two comparison radii by
\begin{equation*}
    T_{\pm}(t)=Qe^{t}(1+\delta)^{\pm1},\quad0\leq t\leq 1,
\end{equation*}
and define the smoothed error terms
\begin{equation}\label{smoothed error}
    E_{Q}^{\pm}(t,z)=N_\delta^{\pm}(Qe^t;z,w)-M_0(T_\pm(t)).
\end{equation}
With the convention \eqref{pm notations}, for $T=Qe^t$ the displayed $M_0$ is exactly the constant term, as we have argued in \eqref{smoothed constant}. Therefore the spectral expansion  is 
\begin{equation}\label{expansion of smoothed error}
    E_{Q}^{\pm}(t,z)=\int_{\Omega'}h_{T_\pm(t)}(\m_\omega)\fpsi(\m_\omega)\overline{u_\omega(z)}u_\omega (w)\,d\omega.
\end{equation}
In the following subsections our goal is to set up the necessary tools for the proof of the following key result. 
\begin{proposition}[The maximal estimate]  \label{maximal estimate}
For every $a\in(0,2)$ and every $\eta>0$
\begin{equation*}
    \left\|\sup_{0\leq t\leq 1}|E_{Q}^{\pm}(t,\cdot)|\right\|_{L^2(X)}\ll_{w,n,a,\eta}Q^{b+\eta}(1+\log Q)
\end{equation*}
with $b=d-\frac{n}{2}$, for both choices of sign.
\end{proposition}

\subsection{A unit-interval Sobolev inequality}

The following Lemma is a key part of the proof of our main theorem, as mentioned already in the introduction. 

\begin{lemma}\label{sobolev inequality}
    If $f\in H^1([0,1];\C),$ with its absolutely continuous representative, then 
    \begin{equation}\label{sobolev 1}
        \sup_{0\leq t\leq 1}|f(t)|^2\leq \int_{0}^{1}|f(t)|^2\,dt+2\left(\int_{0}^{1}|f(t)|^2\,dt\right)^{1/2}\left(\int_{0}^{1}|f'(t)|^2\,dt\right)^{1/2}.
    \end{equation}
Consequently if $F,\partial_tF\in L^2([0,1]\times X)$, then 
    \begin{equation}\label{sobolev 2}
        \left\|\sup_{0\leq t\leq 1}|F(t,\cdot)|\right\|_2^2\leq A_0+2\sqrt{A_0A_1},
        \end{equation}
with        \begin{equation*}A_q=\int_{0}^{1}\|\partial_t^qF(t,\cdot)\|_2^2\,dt,
    \end{equation*}
    where the supremum uses the absolutely continuous representative for almost every point in $X$.
\end{lemma}
\begin{proof}
     There is a $t_0\in[0,1]$ such that $|f(t_0)|^2\leq \int_{0}^{1}|f|^2$ (otherwise the integral would be strictly smaller than every value of its continuous integrand). Applying the fundamental theorem of calculus and the identity $(|f|^2)'=2\Re(f'\bar{f})$, imply that for every $t\in[0,1]$:
    \begin{equation*}
        |f(t)|^2\leq |f(t_0)|^2+2\int_{0}^{1}|f'(s)|\,|f(s)|\,ds.
    \end{equation*}
A straightforward application of Cauchy-Schwarz proves \eqref{sobolev 1}. The same proof applies to $H^1$ functions because they have absolutely continuous representatives and satisfy the fundamental theorem of calculus.

Fubini's theorem for Sobolev functions gives, for almost every $z\in X$, an $H^1$ function $t\to F(t,z)$ with the indicated weak derivative. Applying \eqref{sobolev 1}, integrating in $z$ and using the Cauchy-Schwarz inequality in the $z-$integral for the product of the two square roots, we derive the proof of \eqref{sobolev 2}.
\end{proof}

The usefulness of this elementary estimate is the geometric mean $\sqrt{A_0A_1}$: a factor $L^{-1}$ in $A_0$ can balance a factor $L$ in $A_1$. Bounding a supremum merely by the full $H^1$ norm would miss this balance (see \cite[Remark~3.5]{chatzakosdimakis} for the same phenomenon in the rank one case).

\subsection{The low and intermediate spectral packets}

Let $E_{Q,L}^{\pm}$ denote the integral in \eqref{expansion of smoothed error}, restricted to $\Omega_L'$ defined in \eqref{dyadic shells}. First suppose $L\leq L_0$. Since $T_{\pm}(t)\asymp Q$, we have $U\leq 2L\ll T_{\pm}(t)^{2-\k}$, uniformly in $t$. Proposition \ref{new transform bound} therefore applies for both $q=0,1$. The non constant bound of Proposition \ref{non constant bound} gives 
\begin{equation*}
    d/2+n\b\leq d/2+n\frac{(n-2)}{2}=d-\frac{n}{2}=b.
\end{equation*}
Since $U\geq 1$, we may discard the additional factor $U^{-\beta/2}$. Also $|\fpsi|\leq 1$ and $|\phi_\m|\leq 1$. It follows that 
\begin{equation*}
    |D^q h_{T_\pm(t)}(\m)\hat{\psi}_\delta(\mu)|\ll Q^{b+\eta}L^{q-d/4-1/2}\mathcal{G}(\t),\quad \text{for}\quad q=0,1.
\end{equation*}
Here $D=T_{\pm}\partial_{T_{\pm}}$. Then chain rule gives $\partial_t h_{T_{\pm(t)}}=T_{\pm}(t)\partial_{T_{\pm}(t)}h_{T_\pm}(t)=Dh_{\pm}(t)$.
Taking $\partial_t^q$ to the spectral expansion of $E_{Q,L}^{\pm}$, applying Plancherel contraction for the operator defined in \eqref{conjugated synthesis operator} with $h_{T_\pm}\cdot\psi_\delta$ in place of $h_f$  and using \eqref{first spectral mass on shells} we get 
\begin{eqnarray*}
A_{q,L}&:=&\int_{0}^{1}\|\partial_{t}^{q}E_{Q,L}^{\pm}(t,\cdot)\|_2^2\,dt \notag \\
&\ll& Q^{2b+2\eta}L^{2q-d/2-1}\int_{\Omega_L'}G(\t_{\omega})^2 |u_{\omega}(w)|^2\,d\omega \\
&\ll& Q^{2b+2\eta}L^{2q-1} .   
\end{eqnarray*}
Therefore with the notation in \eqref{sobolev 2} we  
\begin{equation*}
    A_{0,L}\ll Q^{2b+2\eta}L^{-1},\quad A_{1,L} \ll Q^{2b+2\eta}L.
\end{equation*}
Applying Lemma \ref{sobolev inequality} gives 
\begin{equation*}
  \left\|\sup_{t\in [0,1]}|E_{Q,L}^{\pm}(t,\cdot)|\right\|_2\ll Q^{b+\eta},\quad \text{for}\quad L\leq L_0.
\end{equation*}

\subsection{The high spectral tail}

In this part we deal with the tail of the spectrum, for $L>L_0$. In this case the sharp transform is not needed and the coarse bounds of Lemma \eqref{coarse bounds} suffice. Also recall the mollifier bounds of Lemma \eqref{mollifier decay properties} and the spectral mass in shells \eqref{coarse shell bound}, then for any $A>0$ apply the same strategy as in the previous subsection and obtain
\begin{equation*}
    A_{q,L}\ll Q^{2d}L^{d}(\delta L)^{-2A},\quad \text{for}\quad   q=0,1.
\end{equation*}
Therefore Lemma \ref{sobolev inequality} implies
\begin{equation}\label{tail sup estimate}
    \left\|\sup_{t\in[0,1]}|E_{Q,L}^{\pm}(t,\cdot)|\right\|_2\ll Q^{d}L^{(d/2+r)/2}(\delta L)^{-A}.
\end{equation}

The separation between the smoothing cutoff and the endpoint of the sharp range is essential:
\begin{equation*}
    \delta L_0=Q^{2-a-\k}=Q^{\frac{2-a}{2}}= Q^{\k},\quad \k>0.
\end{equation*}
Writing $L=L_0 v$, the right hand side of \eqref{tail sup estimate} is equal to
\begin{equation*}
    Q^{d+(2-\k)(d/2+r)/2-A\k}\,v^{(d/2+r)/2-A}.
\end{equation*}
For $A$ large enough, summation over the dyadic $L>L_0$ yields a convergent series (summing over $v$). For any desired $B>0$ we can increase $A$ further (since it is arbitrary) and thus we can bound the sum by $O_{a,B,n} (Q^{-B})$. This is the only place where we need arbitrarily rapid decay of the mollifier.

\subsection{Convergence and the supremum over all real radii} 
We now supply the functional-analytic details that allow the packet estimates to be used for the actual automorphic kernel. For each packet, its spectral coefficient and its $t$-derivative belong to $L^2([0,1]\times\Omega)$ by the preceding estimates. The bounded synthesis operator therefore implies 
\begin{equation*}
    E_{Q,L}^{\pm}\in H^1([0,1];L^2(X)).
\end{equation*}
We can verify the derivative assertion by testing against a smooth compactly supported scalar function of $t$, integrating by parts in the coefficient space and applying the bounded synthesis map. At this step we do not use any pointwise assertion about individual Eisenstein series as functions in $L^2(X)$.

Choose the scalar absolutely continuous representative given by Fubini for these countably many packets and set 
\begin{equation*}
g_{L}(z)=\sup_{t\in[0,1]}|E_{Q,L}^{\pm}(t,z)|. 
\end{equation*}
Each $g_L$ is measurable: for continuous representatives the supremum is the supremum over the rationals in $[0,1]$. There are $O(1+\log Q)$ packets with $L\leq L_0$ and the high tail is summable. Hence
\begin{equation*}
    \sum_{L}\|g_L\|_2\ll Q^{b+\eta}(1+\log Q).
\end{equation*}
Minkowski's inequality for finite sums and monotone convergence imply 
\begin{equation*}
\sum_{L} g_L \in L^2(X). 
\end{equation*}
Particularly for almost every $z$ the packet series converges absolutely and uniformly for $0\leq t\leq 1$. It defines a continuous function of $t$ at such a point.

For each fixed $t$ the spectral expansion of the smooth convolution kernel identifies the packet sum with \eqref{smoothed error}. This equality initially holds in $L^2(X)$. The actual kernel is square integrable: for fixed $w$ its support in $z$ lies in the projection of 
\begin{equation*}
g_w\text{supp}(\chi_{S_{\pm}(t)}\ast\psi_\delta),
\end{equation*}
(which is a compact subset of $X$) and its periodization is locally finite and smooth in $z$. Subtracting its constant mean preserves square integrability because $\hbox{vol}(X)<\infty$.

The automorphic kernel is also continuous in $t$, since the characteristic function $\chi_T$ depend continuously on $T$ in $L^1(G)$. Convolution with $\psi_\delta$ consequently varies uniformly on compact sets and periodization is locally a finite sum, uniformly for $t\in[0,1]$. The function $M_0$ is continuously differentiable. Thus equality at every rational $t$ outside a countable union of null sets in $z$ extends to every real $t$ by continuity. This justifies \eqref{expansion of smoothed error} in the form needed for a maximal function. Finally the bound
\begin{equation*}
    \sup_{t\in[0,1]}|E_{Q}^{\pm}(t,z)|\leq \sum_{L}g_L(z)
\end{equation*}
implies Proposition \ref{maximal estimate}.
\begin{remark}
Our original counting function has jumps and is not continuous. 
The assertion of continuity used above concerns only the smoothed kernels. Its values at the jump radii is controlled by the pointwise double inequality \eqref{sandwich}, which was proved with closed balls. Thus no exceptional set of radii is introduced when we apply the unsmoothing step.
\end{remark}

\section{The almost-everywhere estimate and unsmoothing}

We can now proceed to the proof of Theorem \ref{ourmainthm}. For this purpose, let $a\in(0,2)$ be fixed and $\zeta>0$. For (sufficiently large) integers $k$ let us also set $Q_k=e^k$ and $\delta_k=Q_k^{-a}$, 

\subsection{A summable family of exceptional sets} For either sign we consider the sets
\begin{equation*}
   \mathcal{E}_{k}^{\pm}=\left\{z\in X:\sup_{0\leq t \leq 1}|E_{Q_k}^{\pm}(t,z)|>Q_k^{b+\zeta}\right\}.
\end{equation*}
Applying the same argument as in the proof of Proposition \ref{maximal estimate} we deduce that the sets $\mathcal{E}_{k}^{\pm}$ are measurable. We apply Proposition \ref{maximal estimate} with $\eta=\zeta/2$. Chebyshev's inequality gives
\begin{equation*}
    \m_{X}(\mathcal{E}_{k}^{\pm})\ll Q_{k}^{-2b-2\zeta}Q_k^{2b+\zeta}(1+\log Q_k)^2=e^{-k\zeta}(1+k^2).
\end{equation*}
Hence, summing over $k$ we get a convergent series. The first Borel-Cantelli lemma (which requires no independence) implies that for almost every $z$, only finitely many of the sets $\mathcal{E}_k^{+}\cup\mathcal{E}_{k}^{-}$ contain $z$. Equivalently, for almost every $z$ there exists a  $k_0(z,w,a,\zeta)$ such that for both signs:
\begin{equation}\label{Borel part}
    \sup_{t\in[0,1]}|E_{Q_k}^{\pm}(t,z)|\leq Q_k^{b+\zeta},\quad \text{for all} \quad k\geq k_0.
\end{equation}
The possible null sets used to identify spectral and geometric representatives for each $k$ are also removed. As a countable union of sets with zero measure, this forms again a set of zero measure.

\subsection{Bounding the error term}
Given any sufficiently large real $T$ we take $k=\lfloor\log T\rfloor$ and we write $T=Q_ke^{t}$ (for $ 0\leq t\leq 1$). Then inequality \eqref{sandwich 2} gives
\begin{equation*}
    N_{\delta_k}^{-}(T;z,w)\leq N_n(T;z,w)\leq N_{\delta_k}^{+}(T;z,w).
\end{equation*}
Subtracting $M_0(T)$ from all three sides, the two exterior terms take the form 
\begin{equation*}
    E_{Q_k}^{\pm}(t,z)+M_0(T(1+\delta_k)^{\pm 1})-M_0(T)
\end{equation*}
respectively. Estimate \eqref{main term approximation} gives
\begin{equation*}
|M_0(T(1+\delta_k)^{\pm 1})-M_0(T)|\ll\delta_kT^{d}\ll T^{d-a}.
\end{equation*}
Together with \eqref{Borel part} this proves the intermediate bound 
\begin{equation*}
    E_0(T;z,w)\ll_{z,w,n,a,\zeta} T^{b+\zeta}+T^{d-a}
\end{equation*}
for all sufficiently large real $T$ and for almost every $z$ and $a,\zeta,w$ fixed. Lemma \ref{ball estimates} adds the factor $M_0(T)-M_n(T)=O_{n,\eta}(T^{d-n/2+\eta})$. We conclude
\begin{equation}\label{true error upper bound before balancing}
    E(T;z,w)\ll T^{b+\zeta}+T^{d-a}+T^{d-n+\eta}.
\end{equation}
These three terms are coming from three different sources: the spectral maxima estimation, the geometric smoothing and the lower terms of the ball volume.

\subsection{Completion of the proof of Theorem \ref{ourmainthm}}

Let us first consider the case $n=3$. We choose $a=3/2$ so that $d-a=b=9/2$. Given $\epsilon>0$, we take $0\leq \zeta\leq \e$ and we choose $\eta>0$ small enough that $d-n+\eta<b+\e$. 
Then \eqref{true error upper bound before balancing} gives
\begin{equation*}
E_3(T;z,w)\ll T^{9/2+\epsilon}
\end{equation*}
for every $\epsilon >0$.

Now let $n\geq 4$. It suffices to prove the claim for $0<\e<1$. We choose $a=2-\epsilon/2$, $\zeta=\e/2$ and $\eta=\e/2$. Then
\begin{eqnarray*}
b+\zeta &\leq& d-2+\epsilon, \\ 
d-a &=& d-2+\epsilon/2, \\
d-n+\eta &\leq& d-2+\epsilon/2.
\end{eqnarray*}
Thus \eqref{true error upper bound before balancing} gives 
\begin{equation*}
E_n(T;z,w)\ll T^{d-2+\e}
\end{equation*}
for every $\epsilon>0$ and for every $n \geq 4$. Notice that the condition $a<2$ holds; since $a=2$ would destroy the positive gap needed for the spectral tail.

So far the full measure set may depend on $\e$. We now apply the proved assertion successively with $\e_j=1/j$ for $j\geq 2$ and we intersect the resulting full measure sets. Their intersection $B_w$ has full measure. Given any $\e>0$, choose $j$ with $1/j<\e$; the bound with $\e_j$ implies the desired bound with $\e$, for $T\geq 1$. This completes the proof of Theorem \ref{ourmainthm}.

\begin{remark}
If we fix the point $z$ instead of $w$ and view the pre-trace formula \eqref{pretrace formula} as a function of $w$, then using the ordinary synthesis \eqref{regular synthesis operator} (instead of the conjugated synthesis) and applying the same spectral bound \eqref{local spectral density} at the fixed point $z$, we deduce Theorem \ref{ourmainthm} for almost every $w$ (depending on $z$).
\end{remark}

\section{A lower bound for the error term}

We conclude with the proof of Theorem \ref{omegaresult}. Although the strength of this $\Omega$-result in analogous with the strength of the known $\Omega$-results for the classical Euclidean and hyperbolic lattice counting problems in large dimensions, it is worth mentioning that the proof in the higher rank case is much more elementary. 

\begin{proof}
For $z= I_n = w$ we have 
\begin{equation*}
N_n (T;I_n, I_n)  = \#\left\{ \gamma  \in \hbox{SL}_n (\mathbb{Z}) : \| \gamma \|_F \leq T \right\}.
\end{equation*}
If we denote $\gamma = (\gamma_{ij})$ then 
\begin{equation*}
\|\gamma\|_F^2 = \sum_{i,j} \gamma_{i,j}^2 \in \mathbb{Z}. 
\end{equation*}
Hence, when $T^2$ stays between two consecutive integers $\ell$ and $\ell+1$, the function $N_n(T; I_n, I_n)$ is constant. Write 
\begin{equation*}
T_1 = \sqrt{\ell+\frac{1}{4}}, \ T_2 = \sqrt{\ell+\frac{3}{4}}.
\end{equation*}
Since $N_n (T_1; I_n, I_n) = N_n (T_2; I_n, I_n)$, using the mean value theorem, by \eqref{sharperror} for some $T \in [T_1, T_2]$ we get
\begin{equation*} 
E_n(T_1; I_n, I_n) - E_n (T_2, I_n, I_n) = c_n \left(T_2^{n(n-1)} - T_1^{n(n-1)}\right) \asymp  T ^{n(n-1)-2}.
\end{equation*}
Hence 
\begin{equation*} 
\max \left\{|E_n(T_2; I_n, I_n)|, |E_n (T_1, I_n, I_n)| \right\} \gg  T^{n(n-1)-2},
\end{equation*}
and the desired result follows.
\end{proof}

\end{document}